\documentclass[12pt]{amsart}
\usepackage{epsfig,comment}
\usepackage{mathtools}
\usepackage[subnum]{cases}
\usepackage{enumerate}
\usepackage{bbm,bm}
\usepackage{amsmath,amssymb,mathrsfs}
\usepackage{color}
\usepackage[bookmarks=true,
bookmarksnumbered=true, breaklinks=true,
pdfstartview=FitH, hyperfigures=false,
plainpages=false, naturalnames=true,
colorlinks=true,pagebackref=true,
pdfpagelabels,colorlinks=true,
	linkcolor=blue,
	citecolor=black,
	filecolor=black,
	urlcolor=black]{hyperref}
\usepackage{wasysym}

\usepackage{pgfplots}
\pgfplotsset{width=10cm,compat=1.9}
\usepgfplotslibrary{fillbetween}

\usepgfplotslibrary{external}
\newtheorem{theorem}{Theorem}[section]
\newtheorem{definition}[theorem]{Definition}
\newtheorem{lemma}[theorem]{Lemma}

\newtheorem{remark}[theorem]{Remark}

\newcommand{\R}{\mathbb{R}} 
\DeclareMathOperator{\LC}{LC} 
\newcommand{\Sp}{\mathbb{S}^{n-1}}

\DeclareMathOperator{\conv}{conv}

\DeclareMathOperator*{\argmax}{argmax}
\DeclareMathOperator{\dom}{dom}
\DeclareMathOperator{\epi}{epi}
\DeclareMathOperator{\hyp}{hyp}
\DeclareMathOperator{\vol}{vol}
\DeclareMathOperator{\proj}{proj}
\DeclareMathOperator{\Gr}{Gr}
\DeclareMathOperator{\lev}{lev}
\DeclareMathOperator{\base}{base}
\DeclareMathOperator{\supp}{supp}

\title{An extremal property of the symmetric decreasing rearrangement}

\author{Steven Hoehner and J\'ulia Novaes}
\date{\today}

\begin{document}

\setcounter{footnote}{0}
\maketitle

\begin{abstract}\noindent
It is shown that for a given log-concave function, its symmetric decreasing rearrangement is always harder to approximate in the symmetric difference metric by inner log-linearizations with a fixed number of break points. This extends a classical result of Macbeath (1951) from convex bodies to a functional setting.
\end{abstract}

\renewcommand{\thefootnote}{}

\renewcommand{\thefootnote}{\arabic{footnote}}
\setcounter{footnote}{0}

\section {Introduction and main results}



In 1951, Macbeath \cite{Macbeath} proved that among all convex bodies of a given volume, the Euclidean ball is hardest to approximate by inscribed polytopes with a fixed number of vertices. More specifically, if $\mathcal{K}^n$ denotes the class of convex bodies in $\R^n$, $K\in\mathcal{K}^n$ and $\mathscr{C}_N^{\rm in}(K):=\{P\in\mathcal{K}^n: P\subset K\text{ is a polytope with at most }N\text{ vertices}\}$, then
\begin{equation}\label{macbeath-ineq}
\inf_{P\in\mathscr{C}_N^{\rm in}(K)}\vol_n(K\setminus P)\leq \inf_{Q\in\mathscr{C}_N^{\rm in}(B_K)}\vol_n(B_K\setminus Q).
\end{equation}
Here $\vol_n$ is the $n$-dimensional Lebesgue measure and $B_K$ is the Euclidean ball centered at the origin  with the same volume as $K$. For results analogous to \eqref{macbeath-ineq} on the mean width approximation of convex bodies by polytopes, we refer the reader to   \cite{BucurFragalaLamboley,Schneider-1967, Schneider1971}. 

The goal of this  note is to extend Macbeath's result from convex bodies to a functional setting. Let $\LC_{\rm c}(\R^n)$ denote the class of log-concave functions on $\R^n$ which are upper semicontinuous and coercive. For $f\in\LC_{\rm c}(\R^n)$, let $f^*$ denote the symmetric decreasing rearrangement of $f$. The quantity $J(f)-J(p)$ is an analogue of the symmetric difference metric in the setting of log-concave functions (see also \cite{Mussnig-Li}). Moreover, for a given log-concave function $f$, an inner log-linearization of $f$ with $N$ break points is the functional analogue of a polytope inscribed in a convex body with $N$ vertices. Let $\mathscr{P}_N(f)$ be the set of inner log-linearizations of $f$ with at most $N$ break points, and let $J(f)=\int_{\R^n}f(x)\,dx$ denote the total mass of $f$ (see Section \ref{background-sec} below for the precise definitions). Our main result, which is the extension of \eqref{macbeath-ineq} to log-concave functions, reads as follows.

\begin{theorem}\label{mainThm} 
Fix integers $n\geq 1$ and $N\geq n+2$, and consider the functional $G_{n,N}:\LC_{\rm c}(\R^n)\to [0,\infty)$ defined by
\[
G_{n,N}(f) := 
\inf_{p\in\mathscr{P}_N(f)}\{J(f)-J(p)\}.
\]
Then for every $f\in\LC_{\rm c}(\R^n)$, we have
\[
G_{n,N}(f) \leq G_{n,N}(f^*).
\]
\end{theorem}

Choosing a specific log-concave function in Theorem \ref{mainThm}, namely the characteristic function $\mathbbm{1}_K$ of a convex body $K\in\mathcal{K}^n$, we will recover \eqref{macbeath-ineq}, as shown in Remark \ref{mainRmk} below. 

To the best of our knowledge, Theorem \ref{mainThm} has not appeared before in the literature. The purpose of this note is to give a direct,  self-contained proof of Theorem \ref{mainThm}. Our simple approach proceeds in two main steps by showing the functional $G_{n,N}$ is monotone under a  Steiner symmetrization and is $L^1$-upper  semicontinuous.  Let us remark here that Theorem \ref{mainThm} can also be derived from the stochastic form of Brunn's principle for $s$-concave measures  by Pivovarov and Rebollo Bueno \cite[Theorem 1.1]{PB-2020}, where a different, probabilistic approach was used. For other works related to Theorem \ref{mainThm}, see, e.g., \cite{BucurFragalaLamboley,Chen-TAMS-2018,Hoehner-2023}.

\subsection{Overview of the paper}

In the next section, we collect the necessary background and notation. In Section  \ref{proof-section}, we state and prove a simple and general extremal result, namely, Theorem \ref{general-thm}, and use it to prove Theorem \ref{mainThm}. Next, in Section \ref{alpha-section} we extend Theorem \ref{mainThm} to all $\alpha$-concave functions $f:\R^n\to[0,\infty)$ (for certain $\alpha\in\mathbb{R}$) satisfying  typical regularity conditions. Finally, in Section \ref{sec:other-apps}, we highlight  further illustrative examples of Theorem \ref{general-thm}.


\section{Background and notation}\label{background-sec}

We will work in $n$-dimensional Euclidean space $\R^n$ with inner product $\langle x,y\rangle=\sum_{i=1}^n x_i y_i$ and norm $\|x\|=\sqrt{\langle x,x\rangle}$. The Euclidean unit ball in $\R^n$ centered at the origin $o$ is denoted $B_n=\{x\in\R^n:\,\|x\|\leq 1\}$. For positive integers $j$ and $m$ with $j\leq m$, the Grassmannian manifold of all $j$-dimensional subspaces of $\R^m$ is denoted $\Gr(m,j)$.

The \emph{convex hull} $\conv A$ of a set $A\subset\R^n$ is the smallest convex set containing $A$, that is, $\conv A=\bigcap \{C: C\text{ is convex and }C\supset A\}$. A \emph{convex body} in $\R^n$ is a convex, compact subset of $\R^n$ with nonempty interior.  The \emph{characteristic function} $\mathbbm{1}_K:\R^n\to\{0,1\}$ of $K\in\mathcal{K}^n$ is defined by 
\[
\mathbbm{1}_K(x) = \begin{cases}
    1, &\text{if }x\in K;\\
    0, &\text{if }x\not\in K.
\end{cases}
\]
In particular, the volume of $K$ is $\vol_n(K)=\int_{\R^n}\mathbbm{1}_K(x)\,dx$.  For more background on convex bodies, we refer the reader to the book  \cite{SchneiderBook} by Schneider.
\subsection{Log-concave functions}

A function $f:\R^n\to[0,\infty)$ is \emph{log-concave} if for every $x,y\in\dom(f)$ and every $\lambda\in[0,1]$, we have
\[
f(\lambda x+(1-\lambda)y) \geq f(x)^\lambda f(y)^{1-\lambda}.
\]
Put another way, this says that the logarithm of $f$ is concave, that is for every $x,y\in\supp(f)=\{x\in\dom(f):\, f(x)\neq 0\}$ and every $\lambda\in[0,1]$,
\[
\log f(\lambda x+(1-\lambda)y) \geq \lambda\log f(x)+(1-\lambda)\log f(y).
\]

Let ${\rm Conv}(\R^n)$ denote the set of all convex functions $\psi:\R^n\to\R\cup\{+\infty\}$. Every  log-concave function $f$ can be expressed in the form $f=e^{-\psi}$ for some $\psi\in{\rm Conv}(\R^n)$. Let $\LC(\R^n)=\{f=e^{-\psi}:\, \psi\in{\rm Conv}(\R^n)\}$ denote the class of log-concave functions on $\R^n$. In particular, if $K\in\mathcal{K}^n$ then $\mathbbm{1}_K$ is log-concave, and a typical way of embedding $\mathcal{K}^n$ into $\LC(\R^n)$ is  via the mapping $K\hookrightarrow\mathbbm{1}_K$.  

Every convex function $\psi\in{\rm Conv}(\R^n)$ is uniquely determined by its \emph{epigraph} $\epi(\psi)=\{(x,t)\in\R^n\times\R:\, t\geq\psi(x)\}$, and every log-concave function $f\in\LC(\R^n)$ is uniquely determined by its \emph{hypograph} $\hyp(f)=\{(x,t)\in\R^n\times\R:\, t\leq f(x)\}$. For any $t\in \R$ and any function $g:\R^n\to\R$,  the \emph{superlevel set} $\lev_{\geq t}g$  is defined as $\lev_{\geq t}g=\{x\in\R^n:\, g(x)\geq t\}$. In particular, for every $K\in\mathcal{K}^n$ we have 
\begin{equation}\label{lev-set-indicator}
\lev_{\geq t}\mathbbm{1}_K = \begin{cases}
    K,  &\text{if }0<t\leq 1;\\
    \varnothing,  &\text{if }t>1.
\end{cases}
\end{equation}

 A convex function $\psi:\R^n\to\R\cup\{+\infty\}$ is \emph{lower semicontinuous} if $\epi(\psi)$ is closed, and $\psi$ is \emph{coercive} if $\lim_{\|x\|\to\infty}\psi(x)=\infty$. The \emph{(effective) domain} $\dom(\psi)$ of $\psi$ is the set $\dom(\psi)=\{x\in\R^n:\,\psi(x)<+\infty\}$. Let ${\rm Conv}_{\rm c}(\R^n)$ denote the set of convex functions $\psi:\R^n\to\R\cup\{+\infty\}$ which are coercive, lower semicontinuous, $\psi\not\equiv +\infty$, $o\in\dom(\psi)$ and $\dom(\psi)$ is not contained in a hyperplane. We also let $\LC_{\rm c}(\R^n):=\left\{f=e^{-\psi}:\,\psi\in{\rm Conv}_{\rm c}(\R^n)\right\}$. Note that every $f\in\LC_{\rm c}(\R^n)$ is upper semicontinuous, coercive (i.e., $\lim_{|x|\to\infty}f(x)=0$) and $o\in\supp(f)$.  Furthermore, if $f=e^{-\psi}\in\LC_{\rm c}(\R^n)$, then  $f$ is integrable (see, for example, \cite[p. 3840]{cordero-erasquin-klartag}). 
 
 The \emph{total mass} $J(f)$ of an integrable function $f:\R^n\to[0,\infty)$  is defined as $J(f)=\int_{\R^n}f(x)\,dx$. It may be regarded as an analogue of volume in the setting of log-concave functions. In particular, for every $K\in\mathcal{K}^n$ we have $J(\mathbbm{1}_K)=\vol_n(K)$.

For more background on log-concave functions, we refer the reader to, e.g., \cite{Colesanti-inbook}.

\subsection{Inner linearizations and piecewise affine approximation of convex functions}\label{inner-linearizations-sec}

In the setting of convex functions, the analogue of a polytope inscribed in a convex body with $N$ vertices is called an inner linearization with $N$ break points. Inner linearizations  have been studied extensively in convex optimization  and stochastic programming; see for example \cite{Bertsekas2011, BL-stochastic-book}.  Let us now describe their construction. Given   $\psi\in{\rm Conv}(\R^n)$  and   a finite point set $\mathbf{X}_N:=\{(x_i,y_i)\}_{i=1}^N\subset\epi(\psi)$ with $x_i\in\dom(\psi)$ and $y_i\geq\psi(x_i)$, consider  the convex hull of the vertical rays with these endpoints:
\begin{equation}\label{union-rays}
\textstyle \conv^{\uparrow}\mathbf{X}_N := \conv\left(\bigcup_{i=1}^N\{(x_i,y): y\geq y_i\}\right).
\end{equation}
Here the  subscript $\uparrow$ indicates that the vertical rays lie above each point $(x_i,y_i)$. The function $p_{\mathbf{X}_N}$ whose epigraph is the set \eqref{union-rays} is called an \emph{inner linearization} of $\psi$ (see \cite{Bertsekas2011, BL-stochastic-book}). 

\begin{definition}\label{inner-linearization-def}
Given $\psi\in{\rm Conv}(\R^n)$ and a finite set of points $\mathbf{X}_N:=\{(x_1,y_1),\ldots,(x_N,y_N)\}\subset\epi(\psi)$, the \emph{inner linearization} $p_{\mathbf{X}_N}:\R^n\to\R$ is defined by
\[
p_{\mathbf{X}_N}(x):=\begin{cases}
\textstyle\min\left\{\sum_{i=1}^N \lambda_i y_i: \substack{\sum_{i=1}^N\lambda_i=1,\,\lambda_1,\ldots,\lambda_N\geq 0,\\\sum_{i=1}^N\lambda_i x_i=x}\right\}, &\text{if }x\in\conv\{x_1,\ldots,x_N\};\\
+\infty, &\text{otherwise}.
\end{cases}
\]
\end{definition}
 In particular, $p_{\mathbf{X}_N}\geq\psi$ and $p_{\mathbf{X}_N}$  is the unique piecewise affine function with $\dom(p_{\mathbf{X}_N})=\conv\{x_1,\ldots,x_N\}$ and $\epi(p_{\mathbf{X}_N})=\conv^{\uparrow}\mathbf{X}_N$. Moreover, $p_{\mathbf{X}_N}\in{\rm Conv}_{\rm c}(\R^n)$. A point $x\in\dom(\psi)$ is called a \emph{break point} of $p_{\mathbf{X}_N}$ if $(x,y)\in\R^n\times \R$ is an extreme point of $\epi(p_{\mathbf{X}_N})$. 
 
 Just as any  convex body $K\in\mathcal{K}^n$ may be approximated by an inscribed polytope $P\subset K$ with at most $N$ vertices in the symmetric difference metric $\vol_n(K)-\vol_n(P)$, a convex function $\psi\in{\rm Conv}_{\rm c}(\R^n)$ may be approximated by an inner linearization $p\geq\psi$ with at most $N$ break points under the total mass difference $J(e^{-\psi})-J(e^{-p})$.  From this perspective, inner linearizations have recently found applications at the interface of Convex Geometry, Functional Analysis and Probability; see \cite{Hoehner-2023,PB-2020,Rinott} and the references therein. Note that for a given integer $N\geq n+2$ and $\psi\in{\rm Conv}_{\rm c}(\R^n)$,  it follows by a compactness argument that there exists a \emph{best-approximating inner linearization} $\widehat{p}\in\mathscr{C}_N(\psi)$ such that $F_N(\psi)=J(e^{-\widehat{p}})$.

Given $\psi\in{\rm Conv}(\R^n)$ and an integer $N\geq n+2$, let 
$\mathscr{C}_N(\psi)$ be the set of all inner linearizations of $\psi$ with at most $N$ break points.  For $f=e^{-\psi}\in\LC_{\rm c}(\R^n)$, we also set $\mathscr{P}_N(f):=\{e^{-p}:\,p\in\mathscr{C}_N(\psi)\}$. A function  $q\in\mathscr{P}_N(f)$ is called an \emph{inner log-linearization} of $f$  (see also \cite{Hoehner-2023,PB-2020,Rinott}). Note that if $q=e^{-p}\in\mathscr{P}_N(f)$, then   $q\leq f$ and $q\in{\rm LC}_{\rm c}(\R^n)$. A point $x\in\supp(f)$ is called a \emph{break point} of $q$ if $x$ is a break point of $p$.

Since the function $x\mapsto e^{-x}$ is a bijection from $\R$ to $(0,\infty)$, one may use inner linearizations of $\psi$ to define a functional version of the convex hull operation for a given set of points in the hypograph of $f=e^{-\psi}$. This hypograph construction has also been  studied in the literature before; see \cite{Hoehner-2023,PB-2020,Rinott}. More generally, one may consider an arbitrary set $X\subset\R^n\times\R$ and define $\conv^\uparrow X=\conv(\cup_{(x,y)\in X}\{(x,z)\in\R^n\times\R:\,z\geq y\})$; see \cite{PB-2020,Rinott}.

\begin{definition}\label{def-f-cvx-hull}
Given $\psi\in{\rm Conv}(\R^n)$ and a  set $S\subset\epi(\psi)$, let $e^{-S}:=\{(x,e^{-y})\in\R^n\times\R:\, (x,y)\in S\}\subset\hyp(f)$.   For $f=e^{-\psi}\in\LC(\R^n)$ and a finite set $\mathbf{Y}_N\subset\hyp(f)$, we define 
\[
\conv_{\downarrow} \mathbf{Y}_N:=e^{-\conv^{\uparrow}\left(-\log \mathbf{Y}_N\right)}
\]
where $-\log \mathbf{Y}_N:=\{(x,-\log y)\in\R^n\times\R:\,(x,y)\in \mathbf{Y}_N\}\subset\epi(\psi)$.     
\end{definition}

Geometrically, this construction means: (i) first map $\mathbf{Y}_N$ to its preimage in $\epi(\psi)$ via the function $(x,y)\mapsto (x,-\log y)$, then (ii) take the convex hull of the vertical rays of the resulting set, then (iii) apply the exponential map $(x,y)\mapsto (x,e^{-y})$ to this set. Its  image, contained in $\hyp(f)$, is $\conv_{\downarrow}\mathbf{Y}_N$. 

Note that for any finite set of points $\mathbf{Y}_N\subset\hyp(f)$, the set $\conv_{\downarrow}\mathbf{Y}_N$ is well-defined and unique.  In particular, if   $\mathbf{Y}_N:=\{(x_1,y_1),\ldots,(x_N,y_N)\}\subset \hyp(f)$, then we have 
\[
\conv_{\downarrow}\mathbf{Y}_N=e^{-\conv\left(\bigcup_{i=1}^N\{(x_i,-\log y):\, y\leq y_i\}\right)}. 
\]

 The construction of $\conv_{\downarrow}\mathbf{Y}_N$ is illustrated below. For simplicity, we have chosen the log-concave function $f(x)=e^{-\psi(x)}$ with $\psi(x)=x^2$.
 \begin{center}
     \includegraphics[scale=0.52]{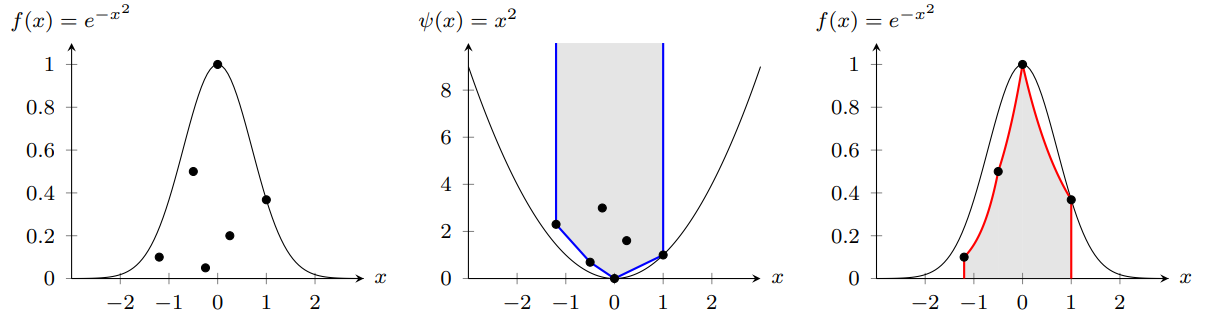}
 \end{center}
 \noindent{\bf\footnotesize Figure 1.} {\footnotesize In the left figure, a finite point set $\mathbf{Y}_N$ is contained in $\hyp(f)$. In the middle figure, the set $\mathbf{Y}_N$ has been mapped to its preimage $-\log\mathbf{Y}_N$ in $\epi(\psi)$. The shaded  region is $\conv_{\uparrow}(-\log\mathbf{Y}_N)$, and the blue curve in its  boundary is the graph of the corresponding inner linearization $p_{-\log\mathbf{Y}_N}$. In the right figure, $\conv_{\uparrow}(-\log\mathbf{Y}_N)$ has been mapped back to its image in $\hyp(f)$; the shaded region there is $\conv_{\downarrow}\mathbf{Y}_N$ and the red curve in its boundary is the graph of the corresponding inner log-linearization $\exp(-p_{-\log\mathbf{Y}_N})$, which has four break points. The area (2-dimensional volume) of this region equals $J(\exp(-p_{-\log\mathbf{Y}_N}))$.} 
 
 \subsection{Symmetric decreasing rearrangements}

 Given a measurable set $A$ in $\R^n$, let $B_A:=\left(\frac{\vol_n(A)}{\vol_n(B_n)}\right)^{1/n}\cdot B_n$ denote the \emph{spherical symmetrization} of $A$, which is the Euclidean ball in $\R^n$ centered at the origin with the same volume as $A$. Given a measurable function $f:\R^n\to[0,\infty)$, the \emph{symmetric decreasing rearrangement} $f^*$ of $f$ may be defined by the relation 
 \[
\text{for all } t\in\R, \quad \lev_{\geq t}f^*=B_{\lev_{\geq t}f}.
 \]
  If $f$ is also integrable, then by the layer-cake principle we have
 \[
J(f^*)=\int_{\R}\vol_n(\lev_{\geq t}f^*)\,dt
=\int_{\R}\vol_n(B_{\lev_{\geq t}f})\,dt
=\int_{\R}\vol_n(\lev_{\geq t}f)\,dt=J(f).
 \]
 
     Let $f:\R^n\to[0,\infty)$ be a measurable function. Then $f^*$ is lower semicontinuous, and is (radially) decreasing in the sense that $\|x\|\leq \|y\|$ implies $f^*(x)\geq f^*(y)$. The function $f^*$ is spherically symmetric, i.e., if   $\|x\|=\|y\|$ then $f^*(x)=f^*(y)$. (Equivalently, $f^*$ is invariant under orthogonal transformations.) Note also that if $f$ is log-concave, then $f^*$ is log-concave.  Moreover, if $p\in[1,\infty)$ and $f\in L^p(\R^n)$, then $f^*\in L^p(\R^n)$ and $\|f\|_p=\|f^*\|_p$. For more background on rearrangements, see, e.g., \cite{burchard2009}.

        
        
        



 \subsection{Steiner symmetrizations}

 Given a nonempty measurable and closed set $A\subset\R^n$ and a hyperplane $H\in\Gr(n,n-1)$, the \emph{Steiner symmetral} $S_H A$ of $A$ with respect to $H$ is given by
\begin{equation}\label{steiner-def-set}
    S_H A = \left\{(x,t)\in (\proj_H A)\times\R:\, |t|\leq \frac{1}{2}(A^+(x)-A^-(x))\right\}
\end{equation}
where for $x\in \proj_H A$ we set 
\begin{align*}
    A^+(x)&:=\sup\{t\in\R:\,(x,t)\in A\}\\ 
    A^-(x)&:=\inf\{t\in\R:\,(x,t)\in A\}. 
\end{align*}

Next, let us recall the definition of a Steiner symmetrization of a measurable function with respect to a hyperplane. First   note that any point $x\in\R^{n+1}$ is uniquely determined by its orthogonal projection $\proj_{\widetilde{H}}x$ in the hyperplane $\widetilde{H}$ and its  signed perpendicular distance $t\in\R$ to $\widetilde{H}$. Given $H=u^\perp\in\Gr(n,n-1)$, where $u$ is a unit vector in $\R^n$,  we 
set $\widetilde{H}:=H\times \R\in\Gr(n+1,n)$.   
Now, given a log-concave function $f:\R^n\to[0,\infty)$, define the functions $f^+,f^-:\proj_{\widetilde{H}}\hyp(f)\to\R$ by
\begin{align*}
    f^+(h)&:=\sup\{t\in\R: h+tu\in\hyp(f)\}\\
    f^-(h)&:=\inf\{t\in\R: h+tu\in\hyp(f)\}.
\end{align*}
In particular, if $f\in\LC_{\rm c}(\R^n)$ then $f^+(h)$ and $f^-(h)$ are finite for all $h\in\proj_{\widetilde{H}}\hyp(f)$ since $\lim_{\|x\|\to\infty}f(x)=0$. Thus, the hypograph of $f\in\LC_{\rm c}(\R^n)$ may be expressed as
\begin{equation}\label{epi-psi}
    \hyp(f)=\left\{(h,t)\in \left(\proj_{\widetilde{H}}\hyp(f)\right)\times\R: \, f^-(h)\leq t\leq f^+(h)\right\}.
\end{equation}
Applying the definition \eqref{steiner-def-set} of the Steiner symmetral of a set, we get 
\begin{equation}\label{steiner-psi}
S_{\widetilde{H}}\hyp(f) = \left\{(h,t)\in \left(\proj_{\widetilde{H}}\hyp(f)\right)\times\R: |t|\leq \frac{1}{2}(f^+(h)-f^-(h))\right\}.
\end{equation}
The \emph{Steiner symmetral} $S_H f$ of $f$ with respect to $H$ may be defined by the relation $\hyp(S_H f)=S_{\widetilde{H}}\hyp(f)$ (see \cite{CGN-2018}). Note that  by Cavalieri's principle, we have $J(f)=J(S_H f)$.

For an ordered sequence of hyperplanes $H_1,\ldots,H_m\in\Gr(n,n-1)$ and a measurable function $f:\R^n\to\R$, we use the notation $\bigcirc_{i=1}^m S_{H_i}f:=S_{H_m}\cdots S_{H_1}f$ to denote the corresponding ordered sequence of iterated Steiner symmetrizations. Any measurable function $f:\R^n\to\R$ can be made to converge to its symmetric decreasing rearrangement $f^*$ via a sequence of Steiner symmetrizations. 

\begin{lemma}\label{convergence-lemma}
   Given $f\in \LC_c(\R^n)$, there exists a sequence of hyperplanes $\{H_i\}_{i=1}^\infty\subset \Gr(n,n-1)$ such that $\bigcirc_{i=1}^m S_{H_i}f$ converges to $f^*$ in the $L^1(\R^n)$ metric as $m\to\infty$.
\end{lemma}
This result follows from, e.g., \cite[Theorem 2]{fortier-thesis}. For more background on Steiner symmetrizations of functions, we refer the reader to, e.g.,  \cite{CGN-2018,fortier-thesis}.
 
\section{Proof of Theorem \ref{mainThm}}\label{proof-section}

We will show that Theorem \ref{mainThm} is a special case of the following general result.

\begin{theorem}\label{general-thm}
Suppose that $F:{\rm LC}_{\rm c}(\mathbb{R}^n)\to (0,\infty)$ is nondecreasing (respectively,  nonincreasing) under Steiner symmetrizations, i.e., $F(S_H f)\geq F(f)$ (resp., $F(S_H f)\leq F(f)$) for every hyperplane $H$ in $\R^n$, and that it is upper semicontinuous (resp., lower semicontinuous) with respect to the $L^1(\mathbb{R}^n)$ metric. Then for every $f\in{\rm LC}_{\rm c}(\mathbb{R}^n)$, we have $F(f)\leq F(f^*)$ (resp.,  $F(f)\geq F(f^*)$). 
\end{theorem}

As a general meta-principle, Theorem \ref{general-thm} is well-known. In particular, several classical integral inequalities can be obtained as special cases of Theorem \ref{general-thm}, including: the  isoperimetric inequality for the total variation of smooth $f$; a second-moment inequality; and a Hardy--Littlewood-type inequality. 

\begin{proof}
    Let $f\in\LC_{\rm c}(\R^n)$. By Lemma  \ref{convergence-lemma}, there exists a sequence of hyperplanes $\{H_i\}_{i=1}^\infty\subset\Gr(n,n-1)$ such that $\bigcirc_{i=1}^m S_{H_i}f \to f^*$ in the $L^1(\R^n)$ metric as $m\to\infty$. Since $F$ is nondecreasing under Steiner symmetrizations, we have
\[
F(f) \leq F(\bigcirc_{i=1}^m S_{H_i}f), \quad m=1,2,3,\ldots
\]
Taking limit superiors as $m\to\infty$ and applying the upper semicontinuity of $F$, we obtain 
\begin{align*}
F(f) &\leq \limsup_{m\to\infty}F(\bigcirc_{i=1}^m S_{H_i}f) \leq F(f^*). 
\end{align*}
Since $f\in\LC_{\rm c}(\R^n)$ was arbitrary, the result follows. 
\end{proof}

We shall prove that the functional $G_{n,N}(f)$ decreases when a Steiner symmetrization is applied to $f$, and that $G_{n,N}$ is lower semicontinuous. This will prove that $G_{n,N}$ satisfies the hypotheses of Theorem \ref{general-thm}, and the result will follow. Structurally, the proof of Theorem \ref{mainThm} follows along the same lines as Macbeath's proof of \eqref{macbeath-ineq} in \cite{Macbeath}, with adaptations made to the setting of log-concave functions.  The main step of Macbeath's proof of \eqref{macbeath-ineq} is showing that a Steiner symmetrization increases the corresponding functional, namely, for every hyperplane $H$ in $\R^n$ we have
\[
\inf_{P\in\mathscr{C}_N^{\rm in}(K)}\vol_n(K\setminus P) \leq \inf_{Q\in\mathscr{C}_N^{\rm in}(S_H K)}\vol_n(S_H K\setminus Q).
\]
Since $\vol_n(K)=\vol_n(S_H K)$, this is equivalent to 
\begin{equation}\label{Macbeath-mainstep}
\sup_{P\in\mathscr{C}_N^{\rm in}(K)}\vol_n(P) \geq \sup_{Q\in\mathscr{C}_N^{\rm in}(S_H K)}\vol_n(Q). 
\end{equation}

Similarly, since $J(f)=J(f^*)$, Theorem \ref{mainThm} may be equivalently formulated as follows.

\begin{lemma}\label{mainLem}
For every $f\in\LC_{\rm c}(\R^n)$ and every $H\in\Gr(n,n-1)$, we have
\[
\sup_{p\in\mathscr{P}_N(f)}J(p) \geq \sup_{q\in\mathscr{P}_N(S_H f)}J(q).
\]
\end{lemma}

\begin{proof}
Choose any $N$ points $x_1,\ldots,x_N\in\hyp(S_H f)$ and let $P_f=\conv_\downarrow\{x_i\}_{i=1}^N$. Each point $x\in\R^{n+1}$ may be expressed as $x=h+t u$, where $h=\proj_{\widetilde{H}}x\in\widetilde{H}$ and $t\in\R$ is the signed perpendicular distance of $x$ to $\widetilde{H}$, which we write in coordinate form as $x=(h,t)\in\widetilde{H}\times\R$.  Define $p_f\in\mathscr{P}_N(S_H f)$ to be the inner log-linearization such that $\hyp(p_f)=P_f$, and set
\begin{align}
    Q_f&:=\conv_\downarrow\left\{\big(h_i,t_i+\tfrac{1}{2}[f^+(h_i)+f^-(h_i)]\big)\right\}_{i=1}^N \label{def-Q-psi}\\
    R_f&:=\conv_\downarrow\left\{\big(h_i,-t_i+\tfrac{1}{2}[f^+(h_i)+f^-(h_i)]\big)\right\}_{i=1}^N.\label{def-R-psi}
\end{align}
Define  inner log-linearizations $q_f$ and $r_f$ of $f$ by $\hyp(q_f)=Q_f$ and $\hyp(r_f)=R_f$. The functions $q_f$ and $r_f$ each have at most $N$ break points, and 
$\proj_{\widetilde{H}}P_f=\proj_{\widetilde{H}}Q_f=\proj_{\widetilde{H}}R_f=\conv\{h_i\}_{i=1}^N$. Since each $(h_i,t_i)\in S_{\widetilde{H}}(\hyp(f))$, we have $|t_i|\leq\frac{1}{2}\left(f^+(h_i)-f^-(h_i)\right)$. Thus by \eqref{steiner-psi} and the definition \eqref{def-Q-psi} of $Q_f$, for each $(h_i,t_i)\in Q_f$ we have 
\begin{align*}
    t_i+\frac{1}{2}\left(f^+(h_i)+f^-(h_i)\right)
    &\leq   \frac{1}{2}\left(f^+(h_i)-f^-(h_i)\right)+\frac{1}{2}\left(f^+(h_i)+f^-(h_i)\right)=f^+(h_i)\\
     t_i+\frac{1}{2}\left(f^+(h_i)+f^-(h_i)\right)
    &\geq   -\frac{1}{2}\left(f^+(h_i)-f^-(h_i)\right)+\frac{1}{2}\left(f^+(h_i)+f^-(h_i)\right)=f^-(h_i).
\end{align*}
Combined with \eqref{epi-psi}, these inequalities show that $Q_f\subset\hyp(f)$. Similarly, $R_f\subset\hyp(f)$. Note that
\begin{align*}
    q_f^+(h_i) &\geq t_i+\frac{1}{2}\left(f^+(h_i)-f^-(h_i)\right),\\
    q_f^-(h_i) &\leq t_i+\frac{1}{2}\left(f^+(h_i)-f^-(h_i)\right),\\
    r_f^+(h_i) &\geq -t_i+\frac{1}{2}\left(f^+(h_i)-f^-(h_i)\right),\\
    r_f^-(h_i) &\leq -t_i+\frac{1}{2}\left(f^+(h_i)-f^-(h_i)\right).
\end{align*}
Hence,
\[
\frac{1}{2}\left(q_f^-(h_i)-r_f^+(h_i)\right) \leq t_i \leq \frac{1}{2}\left(q_f^+(h_i)-r_f^-(h_i)\right)
\]
which shows that each of the points $(h_i,t_i)$ lies in the convex set
\[
T_f:=\conv_\downarrow\left\{(h,t)\in(\proj_{\widetilde{H}}P_f)\times\R: \frac{1}{2}\left(q_f^-(h)-r_f^+(h)\right) \leq t \leq \frac{1}{2}\left(q_f^+(h)-r_f^-(h)\right)\right\}\subset\hyp(f).
\]
Since $P_f=\conv_{\downarrow}\{(h_i,t_i)\}_{i=1}^N$, this implies that  $P_f\subset T_f$. Let $g_f$ be the log-concave function defined by the relation $\hyp(g_f)=T_f$. Then we obtain 
\begin{align*}
    J(p_f) \leq J(g_f)
    &=\int_{\widetilde{H}}\left[g_f^+(h)-g_f^-(h)\right] dh\\
    &=\int_{\proj_{\widetilde{H}}P_f}\left(\frac{1}{2}\left[q_f^+(h)-r_f^-(h)\right]-\frac{1}{2}\left[q_f^-(h)-r_f^+(h)\right]\right)dh\\
    &=\frac{1}{2}\int_{\proj_{\widetilde{H}}Q_f}\left[q_f^+(h)-q_f^-(h)\right]dh
    +\frac{1}{2}\int_{\proj_{\widetilde{H}}R_f}\left[r_f^+(h)-r_f^-(h)\right]dh\\
    &=\frac{1}{2}\int_{\R^n}q_f(x)\,dx+\frac{1}{2}\int_{\R^n}r_f(x)\,dx\\
    &=\frac{1}{2}J(q_f)+\frac{1}{2}J(r_f).
\end{align*}
Hence at least one of $\{J(q_f),J(r_f)\}$ is greater than or equal to $J(p_f)$.
\end{proof}

The final ingredient needed to prove Theorem \ref{mainThm} is the upper semicontinuity of  $G_{n,N}$.

\begin{lemma}\label{cont-lem}
    The functional $G_{n,N}:\LC_{\rm c}(\R^n)\to[0,\infty)$ is upper semicontinuous with respect to the $L^1(\R^n)$ metric. 
\end{lemma}

We will need the following lemma. 
\begin{lemma}\label{J-cont-lem}
If $f=e^{-\psi},\,f_j=e^{-\psi_j}\in\LC_{\rm c}(\R^n)$ and $f_j\to f$ in $L^1(\R^n)$, then
$J(f_j)\to J(f)$ as $j\to\infty$.
\end{lemma}
\begin{proof}
By the definition of the $L^1$ norm, we have
\[
|J(f_j)-J(f)|
=\left|\int_{\R^n}\left(f_j(x)-f(x)\right)dx\right|
\le \int_{\R^n}|f_j(x)-f(x)|\,dx
=\|f_j-f\|_{L^1}\longrightarrow 0
\]
as $j\to\infty$.
\end{proof}

\begin{proof}[Proof of Lemma \ref{cont-lem}]
Since $G_{n,N}=J-S_{n,N}$ and $J(f)=\int_{\R^n}f$ is continuous on $L^1(\R^n)$ when restricted to functions in $\LC_{\rm c}(\R^n)$,
it suffices to show that
\[
S_{n,N}(f):=\sup_{p\in \mathscr{P}_N(f)} J(p) 
\]
is $L^1$-lower semicontinuous on $\LC_{\rm c}(\R^n)$. That is, if $f\in\LC_{\rm c}(\R^n)$ and $\{f_j\}_{j=1}^\infty\subset\LC_{\rm c}(\R^n)$ where $f_j\to f$ in $L^1(\R^n)$, then
\begin{equation}\label{eq:lsc-SN}
\liminf_{j\to\infty} S_{n,N}(f_j)\geq S_{n,N}(f).
\end{equation}

Choose a sequence of ``almost optimal’’ tent functions
\[
p^{(k)} \in \mathscr{P}_N(f)\quad \text{such that}\quad
J(p^{(k)})\geq S_{n,N}(f) - \frac{1}{k}.
\]
Each $p^{(k)}$ has a break point set
\[
\mathbf{X}^{(k)}=\left\{(x^{(k)}_i,y^{(k)}_i)\right\}_{i=1}^{M_k}
\subset \epi(\psi),\quad M_k\leq N,
\]
and we  write $p^{(k)}=e^{-\psi^{(k)}_{\mathbf{X}}}$ for the corresponding inner linearization   
$\psi^{(k)}_{\mathbf{X}}$.

Since $M_k\leq N$, for each fixed $k$ there are at most $N$ break points
$\{x^{(k)}_1,\dots,x^{(k)}_{M_k}\}$.  
Since $f_j\to f$ in $L^1$, we may extract a subsequence of $\{f_j\}$ that
converges almost everywhere.
Then for each fixed $k$, the finite set
$\{x^{(k)}_1,\dots,x^{(k)}_{M_k}\}$ admits a further subsequence  
(of the a.e.\ convergent one) on which
\[
f_{j}(x^{(k)}_i)\longrightarrow f(x^{(k)}_i)
\quad\text{for}\quad 1\leq i\leq M_k.
\]
A diagonal argument over $k$ produces a single subsequence
$\{f_{j_m}\}$ such that for every $k$,
\begin{equation}\label{eq:values-at-nodes}
    f_{j_m}(x^{(k)}_i)
     \longrightarrow
    f(x^{(k)}_i)
    \quad\text{for}\quad
    1\leq i\leq M_k.
\end{equation}
Since $(x^{(k)}_i,t^{(k)}_i)\in\hyp(f)$, we have $t^{(k)}_i>0$; hence
$f(x^{(k)}_i)\geq t^{(k)}_i>0$, so the values are all finite and no boundary issues arise.  
From now on, all limits are taken along this (not relabeled) diagonal subsequence.

Fix $k$ and let $p^{(k)}=e^{-\psi^{(k)}_{\mathbf{X}}}$ with break points 
$\{(x^{(k)}_i,y^{(k)}_i)\}$.  
Define modified heights for $f_j$ by
\[
y^{(k,j)}_i := \max\left\{y^{(k)}_i,\psi_j(x^{(k)}_i)\right\}.
\]
Then $(x^{(k)}_i,y^{(k,j)}_i)\in\epi(\psi_j)$, so the corresponding inner
linearization $\psi^{(k,j)} := \psi_{X^{(k,j)}}$ with
$\mathbf{X}^{(k,j)}:=\{(x^{(k)}_i,y^{(k,j)}_i)\}$ 
satisfies $q^{(k,j)}=e^{-\psi^{(k,j)}}\in \mathscr{P}_N(f_j)$.
Thus,
\[
S_{n,N}(f_j)\geq J(q^{(k,j)})\quad\text{for all $k,j$}.
\]
By \eqref{eq:values-at-nodes}, we have
$\psi_j(x^{(k)}_i)\to\psi(x^{(k)}_i)$ for all $i$, hence
\[
\delta^{(k)}_j := \max_{1\leq i\leq M_k}|y^{(k,j)}_i - y^{(k)}_i|
\longrightarrow 0\quad\text{as}\quad j\to\infty.
\]

By Definition~\ref{inner-linearization-def}, for every $x\in\conv\{x^{(k)}_1,\ldots,x^{(k)}_{M_k}\}$, 
we have
\[
|\psi^{(k,j)}(x) - \psi^{(k)}_{\mathbf{X}}(x)| \leq \delta^{(k)}_j.
\]
Thus, $\psi^{(k,j)} \longrightarrow \psi^{(k)}_{\mathbf{X}}$ uniformly in $x$. 
Exponentiating, we obtain the pointwise convergence $q^{(k,j)}\to q^{(k)}$,  where $q^{(k)}:=e^{-\psi^{(k)}_{\mathbf{X}}}=p^{(k)}$. 
Furthermore, since $\psi^{(k)}_{\mathbf{X}}\geq \psi$, we have $q^{(k)}\leq f$, and hence
\[
0\leq q^{(k,j)}\leq q^{(k)}\leq f\in L^1(\R^n)\cap\LC_{\rm c}(\R^n).
\]
Thus, by the dominated convergence theorem
\begin{equation}\label{eq:conv-Jq}
     J(q^{(k,j)})\longrightarrow J(q^{(k)}).
\end{equation}

Fix $\varepsilon>0$.  
By the definition of the supremum, choose $k$ large enough so that
\[
J(p^{(k)})\geq S_{n,N}(f)-\varepsilon.
\]
Then by \eqref{eq:conv-Jq},
\[
\liminf_{j\to\infty} S_{n,N}(f_j)
\geq 
\liminf_{j\to\infty}J(q^{(k,j)})
=
J(q^{(k)})
\geq S_{n,N}(f)-\varepsilon.
\]
Since $\varepsilon>0$ was arbitrary, we obtain \eqref{eq:lsc-SN}.
\end{proof}


\begin{remark}\label{mainRmk}
Let $K\in\mathcal{K}^n$ be a convex body that (without loss of generality) contains the origin in its interior. By \eqref{lev-set-indicator} we have $\mathbbm{1}_K^*=\mathbbm{1}_{B_K}$, so  choosing $f=\mathbbm{1}_K$ in Theorem \ref{mainThm}, we get $G_{n,N}(\mathbbm{1}_K) \leq G_{n,N}(\mathbbm{1}_{B_K})$. Since $J(\mathbbm{1}_K)=\vol_n(K)=\vol_n(B_K)=J(\mathbbm{1}_{B_K})$, the previous inequality is equivalent to $\sup_{p\in\mathscr{P}_N(\mathbbm{1}_K)}J(p)\geq \sup_{q\in\mathscr{P}_N(\mathbbm{1}_{B_K})}J(q)$. Observe that for a best-approximating inner log-linearization $\widehat{p}\in \argmax_{p\in\mathscr{P}_N(\mathbbm{1}_K)}J(p)$, where $\hyp(\widehat{p})=\conv_{\downarrow}\{(x_i,t_i)\}_{i=1}^N\subset\hyp(\mathbbm{1}_K)=K\times[0,1]$, we must have that  $t_1=\ldots=t_N=1$. Consequently, \[\conv_{\downarrow}\{(x_i,t_i)\}_{i=1}^N=(\conv\{x_1,\ldots,x_N\})\times[0,1]\] is a prism with height 1 and its base $\conv\{x_1,\ldots,x_N\}\in\mathscr{C}_N^{\rm in}(K)$ is a best-approximating polytope of $K$ with at most $N$ vertices. Therefore,
$J(\widehat{p})=\vol_n(\conv\{x_1,\ldots,x_N\})$ and hence $G_{n,N}(\mathbbm{1}_K)=\inf_{P\in\mathscr{C}_N^{\rm in}(K)}\vol_n(K\setminus P)$. From this, we recover Macbeath's result \eqref{macbeath-ineq}.
\end{remark}

\begin{remark}
Very recently in \cite{Hoehner-2023}, the first named author proved the analogue of Theorem \ref{mainThm} for the mean width metric of log-concave functions. The result states that for every log-concave function, its reflectional hypo-symmetrization is always harder to approximate by inner log-linearizations with a fixed number of break points. (For brevity, we refer the reader to \cite{Hoehner-2023} for the specific definitions and precise statement of the result.) 
\end{remark}




\section{Extension to $\alpha$-concave functions}\label{alpha-section}

 Theorem \ref{mainThm} and Definition \ref{def-f-cvx-hull} can be extended to a parameterized family of functions which includes the log-concave functions. Let $\alpha\in[-\infty,+\infty]$ be given. A function $f:\R^n\to[0,\infty)$ is \emph{$\alpha$-concave} if $f$ has convex support $\supp(f)$ and for every $x,y\in\supp(f)$ and every $\lambda\in[0,1]$, it holds that
\begin{equation}
f(\lambda x+(1-\lambda)y) \geq M_\alpha^{(\lambda,1-\lambda)}(f(x),f(y)),
\end{equation}
where for $s,t,u,v>0$ the \emph{$\alpha$-means} $M_\alpha^{(s,t)}(u,v)$ are defined by
\begin{equation}
M_\alpha^{(s,t)}(u,v)=\begin{cases}
(su^\alpha+tv^\alpha)^{1/\alpha}, &\text{if }\alpha\neq 0;\\
u^s v^t, &\text{if }\alpha=0;\\
\min\{u,v\}, &\text{if }\alpha=-\infty;\\
\max\{u,v\}, &\text{if }\alpha=+\infty.
\end{cases}
\end{equation}

\noindent The extreme cases $\alpha\in\{-\infty,0,+\infty\}$ are each understood in a limiting sense. 

Given $\alpha\in[-\infty,+\infty]$, let
\[
\mathcal{C}_\alpha(\R^n)=\left\{f:\R^n\to[0,\infty)\big| \,\substack{f\text{ is }\alpha\text{-concave, upper semicontinuous, integrable,}\\{f\not\equiv 0,\, \dim(\supp(f))=n, \, o\in\supp(f) \text{ and }\lim_{\|x\|\to\infty}f(x)=0}}\right\}.
\]
Note that  if $\alpha_1<\alpha_2$ then $\mathcal{C}_{\alpha_1}(\R^n)\supset\mathcal{C}_{\alpha_2}(\R^n)$ (see \cite{BrascampLieb-1976}). Choosing special values of $\alpha$ yields the following notable subfamilies:
\begin{itemize}
\item $\mathcal{C}_{-\infty}(\R^n)$ is the set of quasiconcave functions. It is the largest set of $\alpha$-concave functions since   $\mathcal{C}_{-\infty}(\R^n)\supset \mathcal{C}_\alpha(\R^n)$ for every $\alpha\in[-\infty,+\infty]$. 

\item $\mathcal{C}_0(\R^n)$ is the set of log-concave functions on $\R^n$.

\item $\mathcal{C}_1(\R^n)$ is the set of concave functions defined on convex sets $\Omega$ and extended by 0 outside of $\Omega$.

\item $\mathcal{C}_{+\infty}(\R^n)$ is the set of multiples of characteristic functions of convex sets in $\R^n$.
\end{itemize}

Let $\alpha\in\R\setminus\{0\}$. Given $f\in\mathcal{C}_\alpha(\R^n)$, the \emph{base function} $\base_\alpha f$ of $f$ is  defined by (see \cite{Rotem2013}) 
\[
\base_\alpha f=\frac{1-f^\alpha}{\alpha}.
\]
In other words, $\base_\alpha f$ is the convex function such that 
\[
\text{for all } x\in\R^n,\quad f(x)=(1-\alpha \base_\alpha f(x))_+^{1/\alpha}
\]
where for $s\in\R$ we set $s_+:=\max\{s,0\}$. Since $f\in\mathcal{C}_\alpha(\R^n)$, we have $\base_\alpha f\in{\rm Conv}_{\rm c}(\R^n)$. In the limiting case $\alpha\to 0$, we have the identity $\base_0 f=-\log f$. In particular, if $K\in\mathcal{K}^n$ then $\base_\alpha\mathbbm{1}_K=I_K^\infty$ for every $\alpha\in\R$, where $I_K^\infty:\R^n\to\R\cup\{+\infty\}$ is the (convex) indicator function of $K$ defined by \begin{equation*}
I_K^\infty(x)=
    \begin{cases}
        0, & \text{if } x \in K;\\
        +\infty, & \text{if } x\not\in K.
    \end{cases}
\end{equation*}
For further relevant properties of $\alpha$-concave functions, see \cite[Proposition 24]{Milman-Rotem}.

Extending Definition \ref{def-f-cvx-hull}, we define the following notion of a convex hull  contained  in the hypograph of an $\alpha$-concave function $f$. To state the definition and the result that follows, let $\mathcal{A}$ denote the set of all $\alpha\in\R$ such that the function $g_\alpha:\R\to[0,\infty)$ defined by $g_\alpha(x):= (1-\alpha x)^{1/\alpha}_+$ is decreasing on its domain (and hence is injective), with the special case $\alpha=0$ being understood in the limiting sense.

\begin{definition}\label{def-f-cvx-hull-alpha}
Let $\alpha\in\mathcal{A}$. Given a convex function $\psi:\R^n\to\R$ and a  set $S\subset\epi(\psi)$, let $(1-\alpha S)_+^{1/\alpha}:=\left\{\left(x,(1-\alpha y)_+^{1/\alpha}\right)\in\R^n\times\R:\, (x,y)\in S\right\}\subset\hyp(f)$.   For $f=(1-\alpha\psi)_+^{1/\alpha}\in\mathcal{C}_\alpha(\R^n)$ with $\psi=\base_\alpha f$ and a finite set $\mathbf{Y}_N\subset\hyp(f)$, we define 
\[
\conv_{\downarrow} \mathbf{Y}_N:=\left[1-\alpha\conv_{\uparrow}\left(\frac{1-(\mathbf{Y}_N)^\alpha}{\alpha}\right)\right]_+^{1/\alpha}
\]
where $\frac{1-(\mathbf{Y}_N)^\alpha}{\alpha}:=\left\{\left(x,\frac{1-y^\alpha}{\alpha}\right)\in\R^n\times\R:\,(x,y)\in \mathbf{Y}_N\right\}\subset\epi(\base_\alpha f)$.     
\end{definition}
If $\mathbf{Y}_N$ is a finite point set in $\hyp(f)$, then the function $p_{\mathbf{Y}_N}\in\mathcal{C}_\alpha(\R^n)$ defined by $\hyp(p_{\mathbf{Y}_N})=\conv_{\downarrow} \mathbf{Y}_N$ is called an \emph{inner $\alpha$-linearization} of $f$. A break point of an inner $\alpha$-linearization is just a break point of its convex base function, which is an inner linearization of $\psi$. Let $\mathscr{P}_{\alpha,N}(f)$ denote the class of all inner $\alpha$-linearizations of $f$ with at most $N$ break points. 

Theorem \ref{mainThm} can be extended as follows.

\begin{theorem}\label{mainThm2} 
Fix integers $n\geq 1$ and $N\geq n+2$, and let $\alpha\in\mathcal{A}$. Consider the functional $G_{n,N,\alpha}:\mathcal{C}_\alpha(\R^n)\to [0,\infty)$ defined by
\[
G_{n,N,\alpha}(f) := 
\inf_{p\in\mathscr{P}_{\alpha,N}(f)}\{J(f)-J(p)\}.
\]
Then for every $f\in\mathcal{C}_\alpha(\R^n)$, we have
\[
G_{n,N,\alpha}(f) \leq G_{n,N,\alpha}(f^*).
\]
\end{theorem}
Theorem \ref{mainThm} is then just the special case  $\alpha=0$. The proof follows along the same lines as that of Theorem \ref{mainThm}, where we apply the same lemmas with the obvious modifications, and use the fact that the real function $x\mapsto (1-\alpha x)_+^{1/\alpha}$ is injective on its support when $\alpha\in\mathcal{A}$.  

\section{Appendix: further applications of Theorem \ref{general-thm}}\label{sec:other-apps}

In this section, we record several well-known classical inequalities which can also be obtained from 
Theorem~\ref{general-thm}. 

\subsubsection{Isoperimetric  inequality for total variation.}
  For smooth compactly supported $f\in \LC_{\rm c}(\R^n)\cap C^1(\R^n)$, the total variation of $f$ can be expressed as
  \[
    F(f):=\int_{\R^n} |\nabla f(x)|\,dx .
  \]
  It is well-known that this functional is nonincreasing under Steiner symmetrizations. Indeed, for any set $E\subset\R^n$ of finite perimeter we have $P(S_H E)\leq P(E)$, where $P$ denotes the perimeter functional (see, e.g., \cite{CCF-annals}). Applying the coarea formula and using the inequality $P(S_H(\{f>t\}))\leq P(\{f>t\})$ for every $t$, we obtain
  \begin{align*}
F(S_H f)=\int_{\R^n}|\nabla S_H f(x)|\,dx &= \int_0^\infty P(S_H(\{f>t\}))\,dt \\
&\leq \int_0^\infty P(\{f>t\})\,dt=\int_{\R^n}|\nabla f(x)|\,dx=F(f).
  \end{align*}
  The lower semicontinuity of $F$ in $L^1$ is also well-known; see, e.g., Theorem 1 in Section 5.2 of \cite{Evans-Gariepy}. Theorem~\ref{general-thm} therefore gives $F(f)\ge F(f^*)$.

  \subsubsection{Second moment (moment of inertia).}
Let
\[
  F(f):=\int_{\R^n} |x|^2 f(x)\,dx .
\]
We prove that: (i) $F(f)<\infty$ for all $f\in\LC_{\rm c}(\R^n)$, (ii) Steiner symmetrization $S_H$ does not increase $F$, and (iii) $F$ is lower semicontinuous with respect to $L^1$ convergence. 

To see (i), let $f=e^{-\psi}\in\LC_{\rm c}(\R^n)$. Then $\psi\in{\rm Conv}_{\rm c}(\R^n)$, so $\psi$ is coercive. Thus, there exist $a,b\in\R$ with $a>0$ such that $\psi(x)\geq a|x|+b$ for all $x\geq R$. Hence,
\[
F(f) = \int_{\R^n} |x|^2 f(x)\mathbbm{1}_{|x|\leq R}(x)\,dx+\int_{\R^n} |x|^2 f(x)\mathbbm{1}_{|x|> R}(x)\,dx.
\]
For the first integral, the set of integration $B_R(0)$ is compact and $f$ is upper semicontinuous, so $f$ is finite on $B_R(0)$. Thus, $\int_{\R^n} |x|^2 f(x)\mathbbm{1}_{|x|\leq R}(x)\,dx$ is finite. For the second integral, using spherical coordinates we get
\[
\int_{\R^n} |x|^2 f(x)\mathbbm{1}_{|x|> R}(x)\,dx\leq e^b\int_{\R^n} |x|^2 e^{-a|x|}\mathbbm{1}_{|x|> R}(x)\,dx=e^b\vol_{n-1}(\partial B_n)\int_R^\infty r^{n+1}e^{-ar}\,dr<\infty. 
\]
Thus, $F(f)<\infty$.

For (ii), fix a hyperplane $H$ with unit normal $u\in\Sp$ and write every $x\in\R^n$ uniquely as
$x=y+tu$ with $y\in H$ and $t\in\R$. Then we have 
\[
  |x|^2=\langle y+tu,y+tu\rangle = |y|^2+2t\langle y,u\rangle +|tu|^2=|y|^2+t^2.
\]
For a measurable set $E\subset\R^n$, let $E_y:=\{t\in\R:\, y+tu\in E\}$ denote its one-dimensional fiber in the $u$-direction. By definition, the Steiner symmetral $S_H E$ satisfies $\vol_1((S_H E)_y)=\vol_1(E_y)$ for every $y\in H$, and $(S_H E)_y$ is a centered interval of the form $\Big[-\tfrac{|E_y|}{2},\,\tfrac{|E_y|}{2}\Big]$ (up to null sets).

For the even convex function $\phi(t):=t^2$ and any measurable set $A\subset\R$ with $\vol_1(A)=m$,
by a standard argument using Jensen's inequality we get
\begin{equation}\label{eq:1D-convex}
  \int_A \phi(t)\,dt \geq \int_{-m/2}^{m/2} \phi(t)\,dt=\int_{S_H A}\phi(t)\,dt.
\end{equation}
Applying \eqref{eq:1D-convex} to the fibers $E_y$ and integrating in $y\in H$, we obtain
\begin{align*}
  \int_{E} |x|^2\,dx
  &= \int_H \left(|y|^2\vol_1(E_y) + \int_{E_y} t^2\,dt\right) d\mathcal H^{n-1}(y)\\
  &\geq \int_H \left(|y|^2\vol_1((S_H E)_y) + \int_{(S_H E)_y} t^2\,dt\right) d\mathcal H^{n-1}(y)
  = \int_{S_H E} |x|^2\,dx .
\end{align*}
Using the layer-cake formula for $f\geq 0$, we get
\[
  f(x)=\int_0^\infty \mathbbm{1}_{\{f>t\}}(x)\,dt \quad \text{and} \quad 
  F(f)=\int_0^\infty \left(\int_{\{f>t\}} |x|^2\,dx\right) dt.
\]
Using this together with the fact that $S_H\{f>t\}=\{S_H f>t\}$, and using the set inequality above holding at each level $t$, we obtain 
\[
  F(S_H f)= \int_0^\infty \int_{\{S_H f>t\}} |x|^2\,dx\,dt
  \leq \int_0^\infty \int_{\{f>t\}} |x|^2\,dx\,dt
  = F(f).
\]

Finally, we prove (iii), i.e., that $F$ is lower semicontinuous in $L^1$. Let $f_k,f\in L^1(\R^n)$ with $f_k\to f$ in $L^1(\R^n)$ and $f_k,f\geq 0$. Fix $R>0$.
Since $|x|^2\mathbbm{1}_{B_R(0)}\in L^\infty(\R^n)$, we have
\[
  \int_{B_R(0)} |x|^2 f_k(x)\,dx \to \int_{B_R(0)} |x|^2 f(x)\,dx
  \quad\text{as}\quad k\to\infty.
\]
Since $F(f_k)\geq \int_{B_R(0)} |x|^2 f_k(x)\,dx$ for all $k\in\mathbb{N}$, taking the limit inferior yields
\[
  \liminf_{k\to\infty} F(f_k) \geq \int_{B_R(0)} |x|^2 f(x)\,dx.
\]
Finally, letting $R\uparrow\infty$ and applying the monotone convergence theorem, we conclude that
\[
  \liminf_{k\to\infty} F(f_k) \geq \sup_{R>0}\int_{B_R(0)} |x|^2 f(x)\,dx
  = \int_{\R^n} |x|^2 f(x)\,dx = F(f).
\]
Therefore, $F$ is lower semicontinuous with respect to $L^1$ convergence.

  \subsubsection{Hardy--Littlewood-type inequality with radial weight.}
  
  Let $g:\R^n\to[0,\infty)$ be radially symmetric,  nonincreasing, and bounded a.e., and set
  \[
    F(f):=\int_{\R^n} f(x)g(x)\,dx .
  \]
  We claim that $F$ is nondecreasing under Steiner symmetrizations and continuous in $L^1$;
  this will imply $F(f)\le F(f^*)$ by Theorem~\ref{general-thm}.

  Fix a hyperplane $H=u^\perp$ and write every $x\in\R^n$ as $x=y+tu$ where $y\in H$ and $t\in\R$. Then as before, $|x|^2=|y|^2+t^2$. By the hypotheses on $g$, we have $g(x)=g(y+tu)=:\omega_{|y|}(|t|)$ where $\omega_r:[0,\infty)\to[0,\infty)$ is even in $t$ and nonincreasing in $|t|$ (since $g$ depends only on $|x|$ and is radially nonincreasing). For a nonnegative $f$, define the fiber function $f_y(t):=f(y+tu)$. The Steiner symmetral $S_H f$ is obtained by replacing, for a.e. $y\in H$, the fiber $f_y$ by its one-dimensional symmetric decreasing rearrangement $f_y^*$ (preserving the one-dimensional integral on each fiber). The one-dimensional Hardy--Littlewood inequality with an even, decreasing weight function yields, for a.e. $y\in H$,
\[
\int_{\R^n}f_y(t)\omega_{|y|}(|t|)\,dt\leq \int_{\R^n}f_y^*(t)\omega_{|y|}(|t|)\,dt.
\]
Integrating over $y\in H$ and applying Fubini's theorem, we get
\begin{align*}
    F(f) =\int_{\R^n}f(x)g(x)\,dx&=\int_H\int_{\R}f_y(t)\omega_{|y|}(|t|)\,dt\\
    &\leq \int_H\int_{\R}f_y^*(t)\omega_{|y|}
    (|t|)\,dt
    =\int_{\R^n}(S_H f)(x)g(x)\,dx=F(S_H f).
\end{align*}

To establish the continuity of $F$ in $L^1$, note that by hypothesis $g\in L^\infty(\R^n)$ and hence for any sequence $f_k\to f$ in $L^1(\R^n)$, we have
\begin{align*}
    \left|F(f_k)-F(f)\right| =\left|\int_{\R^n}(f_k(x)-f(x))g(x)\,dx\right|
    \leq\|g\|_{L^\infty}\|f_k-f\|_{L^1}\longrightarrow 0\quad\text{as}\quad k\to\infty.
\end{align*}
\section*{Acknowledgments}

We would like to thank the Perspectives on Research In Science \&  Mathematics (PRISM) program at Longwood University for its  support. SH would  like to thank Michael Roysdon and Sudan Xing for the discussions. We are also indebted to the anonymous referee(s) for their careful reading of the manuscript and many helpful suggestions.

\bibliographystyle{plain}
\bibliography{main}

@article{Bertsekas2011,
author={Bertsekas, D. P. and Yu, H.},
title={A unifying polyhedral approximation framework for convex optimization},
journal={SIAM Journal on Optimization},
year={2011},
volume={21},
number={1},
pages={333--360}
}

@book{BL-stochastic-book,
author={Birge, J. R. and Louveaux, F.},
title={Introduction to Stochastic Programming},
edition={2},
publisher={Springer},
year={2011}
}

@article{BrascampLieb-1976,
author={Brascamp, H. J. and Lieb, E. H.}, 
title={On extensions of the {B}runn-{M}inkowski and {P}r\'ekopa-{L}eindler theorems, including inequalities for log concave functions, and with an application
to the diffusion equation},
journal={Journal of Functional Analysis}, 
volume={22},
number={4},
pages={366--389}, 
year={1976}
}

@article{BucurFragalaLamboley,
author={Bucur, D. and Fragal\`a, I. and Lamboley, J.},
title={Optimal convex shapes for concave functionals},
journal={{ESIAM}: Control, Optimisation and Calculus of Variations},
volume={18},
year={2012},
issue={3},
pages={693--711}
}

@article{Burchard2009,
author={Burchard, A.},
year={2009},
title={A short course on rearrangement inequalities},
journal={Available online at
{\tt http://www.math.utoronto.ca/almut/rearrange.pdf}}
}

@article{Chen-TAMS-2018,
author={Chen, T.},
journal={Transactions of the American Mathematical Society},
volume={370},
number={7},
title={On some determinant and matrix inequalities with a geometrical flavour},
year={2018},
pages={5179--5208}
}

@article{CCF-annals,
author={Chleb\'ik, M. and Cianchi, A. and Fusco, N.},
title={The perimeter inequality under {S}teiner symmetrization: {C}ases of equality},
year={2005},
volume={162},
pages={525--555},
journal={Annals of Mathematics}
}

@incollection{Colesanti-inbook,
title={Log-Concave Functions},
author={Colesanti, A.},
year={2017},
booktitle={Convexity and Concentration},
editor={Carlen, E. and Madiman, M. and Werner, E.},
series={The IMA Volumes in Mathematics and its Applications},
volume={161},
publisher={Springer},
address={New York},
pages={487--524}
}

@article{CGN-2018,
author={Colesanti, A. and Saor\'in G\'omez, E. and Nicol\'as, J. Y.},
title={On a {L}inear {R}efinement of the {P}r\'ekopa-{L}eindler {I}nequality},
journal={Canadian Journal of Mathematics},
year={2018},
volume={68},
numer={4},
pages={762--783}
}

@article{cordero-erasquin-klartag,
author={Cordero-Erausquin, D. and Klartag, B.},
title={Moment measures},
journal={Journal of Functional Analysis},
year={2015},
volume={268},
pages={3834--3866}
}

@incollection{Evans-Gariepy,
author={Evans, L. C. and Gariepy, R.},
title={Measure {T}heory and {F}ine {P}roperties of {F}unctions}, 
series={Studies in Advanced Mathematics}, 
publisher={CRC Press}, 
address={Boca Raton, FL}, 
year={1992}
}

@mastersthesis{fortier-thesis,
	author       = {Fortier, M.}, 
	title        = {Convergence {R}esults for {R}earrangements: {O}ld and {N}ew},
	school       = {The University of Toronto},
	year         = {2010},
	month        = {December}
}

@article{Hoehner-2023,
author={Hoehner, S.},
title={On {M}inkowski symmetrizations of $\alpha$-concave functions and related applications},
year={2025},
journal={arXiv:2301.12619}
}

@article{Macbeath,
author={Macbeath, A. M.},
title={An extremal property of the hypersphere},
journal={Mathematical Proceedings of the Cambridge Philosophical Society},
volume={47},
number={1},
year={1951},
pages={245--247}
}

@article{Milman-Rotem,
year={2013},
title={Mixed integrals and related inequalities},
author={Milman, V. and Rotem, L.},
journal={Journal of Functional Analysis},
volume={264},
number={2},
pages={570--604}
}

@article{Mussnig-Li,
author={Li, B. and Mussnig, F.},
year={2022},
title={Metrics and {I}sometries for {C}onvex {F}unctions},
journal={International Mathematics Research Notices},
volume={2022},
number={18},
pages={14496--14563}
}

@article{Rotem2013,
author={Rotem, L.},
title={Support functions and mean width for $\alpha$-concave functions},
journal={Advances in Mathematics},
year={2013},
volume={243},
pages={168--186}
}

@article{Schneider-1967,
author={Schneider, R.},
title={Eine allgemeine {E}xtremaleigenschaft der {K}ugel},
journal={Monatshefte f\"ur Mathematik},
volume={71},
year={1967}, 
pages={231--237}
}

@article{PB-2020,
author={Pivovarov, P. and Rebollo Bueno, J.},
title={Stochastic forms of {B}runn's principle},
journal={arXiv:2007.03888},
year={2020}
}

@article{Rinott,
author={Rinott, Y.},
title={On convexity of measures},
journal={The Annals of Probability},
year={1978},
volume={4},
number={6},
pages={1020--1026}
}

@article{Schneider1971,
author={Schneider, R.},
year={1971},
title={Zwei {E}xtremalaufgaben f\"ur konvexe {B}ereiche},
journal={Acta Mathematica Academiae Scientiarum Hungaricae},
volume={22},
pages={379--383}
}

@book{SchneiderBook, 
place={Cambridge}, 
edition={2}, 
series={Encyclopedia of Mathematics and its Applications}, 
title={Convex Bodies: The Brunn–Minkowski Theory},  publisher={Cambridge University Press}, author={Schneider, R.}, 
year={2013}, 
collection={Encyclopedia of Mathematics and its Applications}
}


\vspace{3mm}

\noindent {\sc Department of Mathematics \& Computer Science, Longwood University, U.S.A.}

\noindent {\it E-mail address:} {\tt hoehnersd@longwood.edu; jlia.novaes@live.longwood.edu}

\end{document}